\documentclass[letterpaper]{amsart}
\usepackage[utf8]{inputenc}

\usepackage{amsfonts}
\usepackage{amssymb}
\usepackage{tikz}
\usepackage{enumerate}
\usepackage{xcolor}
\usepackage{graphicx}
\usepackage{microtype}
\usepackage{mathrsfs}

\usepackage{amsaddr}

\theoremstyle{plain}
\newtheorem{thm}{Theorem}

\newtheorem{cor}[thm]{Corollary}

\theoremstyle{definition}
\newtheorem{defn}{Definition}
\newtheorem{ex}{Example}

\theoremstyle{remark}

\DeclareMathOperator{\rk}{rk}

\DeclareMathOperator{\dhr}{DHR}

\newcommand{\graph}{G}

\title{Tropical geometry of Rado matroids}
\author{Calum~Buchanan \and Richard~Danner}
\address{Dept.\:of Mathematics \& Statistics \\ University of Vermont \\ Burlington, VT, USA}
\email{$\{\mbox{calum.buchanan}, \mbox{richard.danner}\}$@uvm.edu}
\date{February 19, 2024}

\begin{document}
\begin{abstract}
In this note, we characterize the products of simplicial generators for the Chow ring of a loopless matroid, extending a result of Backman, Eur, and Simpson. We prove that the stable intersection of a collection of tropical hyperplanes centered at the origin with the Bergman fan of a matroid is the Bergman fan of the dual of a certain Rado matroid.
\end{abstract}

\keywords{Tropical geometry, Chow ring, Rado matroid} 
\subjclass{05B35 (Primary), 05B40, 14T05, 14C17 (Secondary)}

\maketitle

The Chow ring $A^\bullet(M)$ of a loopless matroid $M$ on ground set $E$ was introduced by Feichtner and Yuzvinsky in~\cite{feichtner2004chow} as a generalization of the cohomology ring of De Concini and Procesi's wonderful compactification of the complement of a hyperplane arrangement~\cite{de1995wonderful}.
The importance of the Chow ring was demonstrated by Adiprasito, Huh, and Katz in the proof of the Heron-Rota-Welsh conjecture~\cite{adiprasito2018hodge}.
Feichtner and Yuzvinsky define $A^\bullet(M)$ to be the graded ring
$\mathbb{R}[z_F \mid F \in \mathscr{L}_M\backslash \{\varnothing\}]$ modulo the ideals $\langle z_F z_{F'} \mid F,F'~\text{incomparable} \rangle$ and $\langle \sum_{F \supseteq a}z_F \mid a \in \mathfrak{A}_M \rangle$,
where $\mathscr{L}_M$ denotes the lattice of flats of $M$ and $\mathfrak{A}_M$ the set of atoms in $\mathscr{L}_M$.

In 2019, Backman, Eur, and Simpson introduced a set of generators, called {\em simplicial generators}, giving a new presentation of $A^{\bullet}(M)$~\cite{backman2023simplicial}.
The simplicial generators are defined, for each nonempty subset $A$ of $E$, by $h_A = - \sum_{T \supseteq A} z_T(M) \in A^1(M)$.
We denote by $A_{\nabla}^{\bullet}(M)$ the presentation of the Chow ring of $M$ by the simplicial generators: 
\[A_{\nabla}^{\bullet}(M) = \mathbb{R}[h_F \mid F \in \mathscr{L}_M \backslash \{\emptyset \}] / (I+J), \]
where $I=\langle h_a \mid a\in \mathfrak{A}_M\rangle$ and $J=\langle (h_F-h_{F \vee F'})(h_{F'}-h_{F \vee F'}) \mid F,F'\in \mathscr{L}_M \rangle$.
We note that this presentation of $J$, appearing in~\cite{larson2022k}, differs from that of~\cite{backman2023simplicial}.

The simplicial presentation lends itself to the following combinatorial interpretation of the Chow ring of $M$.
A simplicial generator $h_A$ corresponds (via a combinatorial analogue of the cap product) to the {\em principal truncation of $M$ at $F$}, denoted $T_F(M)$, where $F$ is the smallest flat of $M$ containing $A$; this is the matroid with bases $B - f$ over all bases $B$ of $M$ which intersect $F$ nontrivially, and over all $f \in B \cap F$~\cite[Theorem~3.2.3]{backman2023simplicial}.
Furthermore, the cap product allows for a bijection between the monomial basis for $A^c_\nabla (M)$, 
\[\left\{h_{F_1}^{a_1} \cdots h_{F_k}^{a_k} \mid \textstyle{\sum} a_i = c, \ \varnothing = F_0 \subsetneq F_1 \subsetneq \cdots \subsetneq F_{k}, \ 1 \leq a_i < \rk_M(F_i) - \rk_M(F_{i-1}) \right\},\]
and a class of matroids called {\em relative nested quotients of $M$}.
We will forego the definition of these matroids until after defining their generalizations which appear in Theorem~\ref{thm:coradointersections}.
Before stating the theorem, we recall notations and definitions of certain matroid operations.

For a matroid $M$ on ground set $E$, we let $\mathcal{I}(M)$, $\mathcal{B}(M)$, and $\mathcal{C}(M)$ denote the collections of independent sets, bases, and circuits of $M$, respectively.
The rank in $M$ of a subset $S$ of $E$ is denoted by $\rk_M(S)$.
The {\em uniform matroid} $U_{k, E}$ is the matroid whose bases are all of the $k$-subsets of $E$.
The {\em dual} of $M$, denoted $M^*$, is the matroid on $E$ whose bases are the complements of the bases of $M$.
The Bergman fan $\Sigma_M$ of $M$ is the polyhedral fan in $\mathbb{R}^E/\langle e_E \rangle$ consisting of the cones $\operatorname{cone}\{ e_F \mid F \in \mathcal{F} \}$ for each flag $\mathcal{F}=\{ \varnothing \neq F_1 \subsetneq \cdots \subsetneq F_t \neq E \}$ of flats of $M$, where $e_F$ denotes the indicator vector for $F$.

For a subset $S$ of $E$, let $H_S$ denote the corank-$1$ matroid on $E$ with collection of bases $\{ E - s \mid s \in S \}$.
It is well-known that the Bergman fans of corank-$1$ matroids are precisely the tropical hyperplanes centered at the origin.
Speyer~\cite{speyer2008tropical} defined a notion of stable intersection for tropical linear spaces, and it is noted in~\cite{hampe2017intersection} that, as a special case of Theorem~4.11 in the former paper, a stable intersection of Bergman fans is the Bergman fan of a matroid intersection, defined as follows.
The {\em matroid intersection} of matroids $M$ and $N$ on a common ground set $E$, denoted $M \wedge N$, is the matroid whose spanning sets are the intersections of the spanning sets of $M$ and $N$.
A related notion is that of {\em matroid union}, defined in~\cite{nash1966application}: $M \vee N$ is the matroid whose independent sets are of the form $I \cup J$ for $I \in \mathcal{I}(M)$ and $J \in \mathcal{I}(N)$.
We will make use of the fact that $M \wedge N$ can equivalently be defined as $(M^* \vee N^*)^*$.

We characterize the products of simplicial generators in $A^\bullet_\nabla(M)$ using duals of Rado matroids.
In order to define a Rado matroid, we recall Rado's theorem.

\begin{thm}[Rado's theorem~\cite{rado1942theorem}]
Let $M$ be an arbitrary matroid on a set $Y$, and let $\mathcal{X}$ be a collection of subsets $X_1, \ldots, X_m$ of $Y$.
There exists a transversal of $\mathcal{X}$ which is independent in $M$ if and only if $\rk_M(\cup_{j\in J} X_j) \geq |J|$ for all $J \subseteq \{1, \ldots, m\}$.
\end{thm}

Using Rado's theorem, it is not hard to show that the subsets of $\mathcal{X}$ which have independent transversals in $M$ form the independent sets of a matroid on $\mathcal{X}$.

\begin{defn}\label{defn:RadoMatroid}
Given a matroid $M$ on ground set $Y$ and a multiset $\mathcal{X}$ of subsets $X_1, \ldots, X_m$ of $Y$, the {\em Rado matroid induced by $\mathcal{X}$ and $M$} is the matroid with ground set $\mathcal{X}$ and, for independent sets, those subcollections of $\mathcal{X}$ with a transversal which is independent in $M$.
Abusing terminology slightly, given a bipartite graph $H$ with partition $(\mathcal{X}, Y)$, the subsets of $\mathcal{X}$ which are matched to a set in $\mathcal{I}(M)$ form the independent sets of the {\em Rado matroid induced by $H$ and $M$}, denoted $R_{H,M}$. 
\end{defn}

When $M$ is the free matroid $U_{|Y|, Y}$, Rado's theorem specializes to Hall's theorem on transversals, and we obtain transversal matroids from Rado matroids.
We now introduce a class of graphs to be used in our characterization of products of simplicial generators in $A^\bullet_\nabla(M)$.

\begin{defn}\label{defn:Gamma}
Let $\mathcal{A}$ be a collection of subsets $A_1, \ldots, A_m$ of a finite set $E$, and let $\hat{E}$ be a copy of $E$, $\hat{E} = \{ \hat{e} \mid e \in E \}$.
We define $\graph(\mathcal{A})$ to be the graph with bipartition $(E, \hat{E} \cup \mathcal{A})$ and edge set
\[\{ e \hat{e} \mid e \in E \} \cup \{eA \mid A \in \mathcal{A} \mbox{ and } e \in A\}.\]
\end{defn}

Figure~\ref{fig:Rado} depicts an example of such a graph.
We remind the reader that $\graph(\mathcal{A})$ is {\em not} the graph typically used to represent the set system $\mathcal{A}$, which was denoted by $H$ in Definition~\ref{defn:RadoMatroid}.

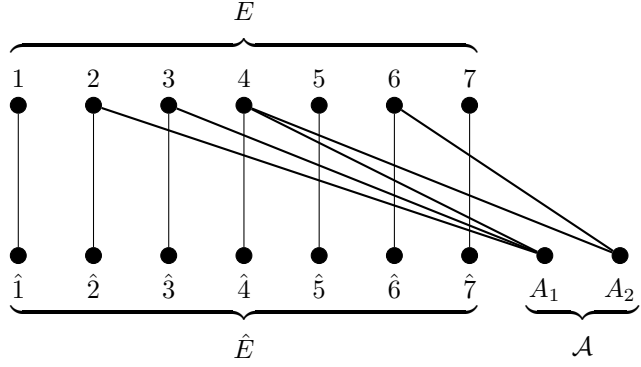
\begin{figure}[ht]
    \centering
    \begin{tikzpicture}
    [every node/.style={circle, draw=black!100, fill=black!100, inner sep=0pt, minimum size=6pt}]

    \foreach \i in {1,...,7}
    {
    \node (\i) at (\i, 2) [label={[yshift=.1cm]$\i$}] {};
    \node (b\i) at (\i, 0) [label={[yshift=-.7cm]$\hat{\i}$}] {};
    \draw (\i) edge (b\i);
    }

    \node (A1) at (8,0) [label={[yshift=-.83cm]$A_1$}] {};
    \node (A2) at (9,0) [label={[yshift=-.83cm]$A_2$}] {};

    \foreach \i in {2,3,4}
    {\draw[thick] (A1) edge (\i);}

    \foreach \i in {4,6}
    {\draw[thick] (A2) edge (\i);}

    \node[rectangle, fill=none,draw=none] (Ebrace) at (4,2.8) {$\overbrace{\hspace{6.2cm}}$};
    \node[fill=none,draw=none] (E) at (4, 3.25) {$E$};

    \node[rectangle, fill=none,draw=none] (hatEbrace) at (4,-.75) {$\underbrace{\hspace{6.2cm}}$};
    \node[fill=none,draw=none] (hatE) at (4, -1.2) {$\hat{E}$};

    \node[rectangle, fill=none,draw=none] (Abrace) at (8.5,-.75) {$\underbrace{\hspace{1.5cm}}$};
    \node[fill=none,draw=none] (A) at (8.5, -1.2) {$\mathcal{A}$};

    \end{tikzpicture}
    \caption{The graph $\graph(\mathcal{A})$, where $E = \{1, \ldots, 7\}$, $\mathcal{A} = \{A_1, A_2\}$, $A_1 = \{2,3,4\}$, and $A_2 = \{4,6\}$.}
    \label{fig:Rado}
\end{figure}

\begin{thm}\label{thm:coradointersections}
Let $\mathcal{A}$ be a collection of (not necessarily distinct) subsets $A_1, \ldots, A_m$ of a set $E$, let $M$ be a matroid on $E$, and let $\graph = \graph(\mathcal{A})$. We have
\[M \wedge H_{A_1} \wedge \cdots \wedge H_{A_m} = (R_{\graph, N})^*,\]
where $N$ is the matroid $\hat{M}^* \oplus U_{m, \mathcal{A}}$ on $\hat{E} \cup \mathcal{A}$, and $\hat{M}^*$ is a copy of $M^*$ on $\hat{E}$.
\end{thm}

Before proving Theorem~\ref{thm:coradointersections}, we note a corollary obtained via the correspondence between the simplicial generators $h_S$ in $A^\bullet_\nabla(M)$ and matroid intersections of the form $M \wedge H_S$, as described in Section~3 of \cite{backman2023simplicial}.

\begin{cor}
If $\mathcal{A}$ is a collection of subsets $A_1, \ldots, A_m \subseteq E$ and $M$ is a loopless matroid on $E$, then the product of simplicial generators $h_{A_1}\cdots h_{A_m}$ is the Bergman class of the matroid $(R_{\graph(\mathcal{A}), N})^*$ from Theorem~\ref{thm:coradointersections}.
\end{cor}

\begin{proof}[Proof of Theorem~\ref{thm:coradointersections}]
We work by induction on $m$.
Since $M \wedge H_{A_1} \wedge \cdots \wedge H_{A_m} = (M^* \vee H_{A_1}^* \vee \cdots \vee H_{A_m}^*)^*$, it suffices to show that $M^* \vee H_{A_1}^* \vee \cdots \vee H_{A_m}^* = R_{\graph, N}$.
The base case, $m = 0$, is trivial.
Let $m \geq 1$, and assume that the result holds for the collection $\mathcal{A} - A_m$; that is,
\[ R' := M^* \vee H_{A_1}^* \vee \cdots \vee H_{A_{m-1}}^* \]
is the Rado matroid on $E$ induced by $\graph(\mathcal{A} - A_m)$ and $\hat{M}^* \oplus U_{m-1, \mathcal{A}-A_m}$.
By the associativity of matroid union, it suffices to show that the independent sets of $R' \vee H_{A_m}^*$ are precisely the independent sets of $R_{\graph, N}$.

First, suppose that a subset $I$ of $E$ is independent in $R' \vee H_{A_m}^*$.
We will show that $I$ is matched in $\graph$ to an independent set in the matroid $\hat{M}^* \oplus U_{m, \mathcal{A}}$ on $N$.
By definition, $I = J \cup K$, where $J \in \mathcal{I}(R')$ and $K \in \mathcal{I}(H_{A_m}^*) = \mathcal{I}(U_{1, A_m})$.
If $K \subseteq J$, then $I$ is independent in $R'$, and thus $I$ is independent in $R_{\graph, N}$.
Otherwise, we have $K = \{a\}$ for some $a \notin J$.
Take a matching from $J$ to an independent set of $R'$ and add the edge $a A_m$, which is clearly disjoint from the others, to obtain a matching from $I$ to an independent set in $R_{\graph, N}$.

Second, suppose that $I \subseteq E$ is matched in $\graph$ to a subset $L$ of $\hat{E} \cup \mathcal{A}$ which is independent in $\hat{M}^* \oplus U_{m, \mathcal{A}}$.
If $A_m \notin L$, then $I$ is matched in $\graph$ to an independent set in $\hat{M}^* \oplus U_{m-1, \mathcal{A}-A_m}$ on $\hat{E} \cup \{A_1, \ldots, A_{m-1}\}$; that is, $I \in \mathcal{I}(R')$, and thus $I \in \mathcal{I}(R' \vee H_{A_m}^*)$.
Otherwise, if $A_m \in L$, then there exists some $a$ in $A_m \cap I$ such that $I - a$ is matched to an independent set in $\hat{M}^* \oplus U_{m-1, \mathcal{A}-A_m}$ on $\hat{E} \cup \{A_1, \ldots, A_{m-1}\}$.
Thus, $I - a \in \mathcal{I}(R')$, which implies that $I\in \mathcal{I}(R' \vee H_{A_m}^*)$.
We have shown that $\mathcal{I}(R' \vee H_{A_m}^*) = \mathcal{I}(R_{\graph, N})$, which completes the proof.
\end{proof}

\begin{ex}\label{ex:R*}
Let us again consider the graph $\graph = \graph(\mathcal{A})$ in Figure~\ref{fig:Rado}.
Let $M$ be the graphic matroid for the graph $H$ with edge set $E$ depicted below.
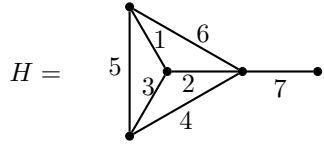
\[
\begin{tikzpicture}
        [every node/.style={circle, draw=black!100, fill=black!100, inner sep=0pt, minimum size=3pt}]

        \node[draw=none, fill=none] (H) at (-1.75, 0) {$H = $};
        
        \node (0) at (0,0) {};
        \node (1) at (-.5,.86) {};
        \node (2) at (1,0) {};
        \node (3) at (-.5,-.86) {};
        \node (4) at (2, 0) {};
        \draw[thick] (0) -- (1) node [midway, right, draw=none, fill=none] {1};
        \draw[thick] (0) -- (2) node [near start, below, draw=none, fill=none] {2};
        \draw[thick] (0) -- (3) node [midway, above=2pt, draw=none, fill=none] {3};
        \draw[thick] (2) -- (3) node [midway, below=2pt, draw=none, fill=none] {4};
        \draw[thick] (3) -- (1) node [midway, above=2pt, left=1pt, draw=none, fill=none] {5};
        \draw[thick] (1) -- (2) node [midway, above=2pt, right=2pt, draw=none, fill=none] {6};
        \draw[thick] (2) -- (4) node [midway, below=2pt, draw=none, fill=none] {7};
\end{tikzpicture}
\]
The matroid $M$ has rank $4$. Letting $N = \hat{M}^* \oplus U_{2, \mathcal{A}}$ and $R = R_{\graph, N}$, we have $M \wedge H_{A_1} \wedge H_{A_2} = R^*$ by Theorem~\ref{thm:coradointersections}.
The Rado matroid $R^*$ is a rank-$2$ matroid with set of bases $\{17, 27, 37, 47, 57, 67\}$.
For instance, $\{1,2,3,5,6\} \subset E$ is matched in $\graph$ to $\{\hat{1}, \hat{3}, \hat{5}, A_1, A_2\} \in \mathcal{B}(M^* \oplus U_{2, \mathcal{A}})$, and so $\{4,7\} \in \mathcal{B}(R^*)$.
\end{ex}

The duals of transversal matroids are known as {\em strict gammoids}.
We refer to the dual of a Rado matroid as a {\em coRado matroid}.
As we noted earlier, it is known that an intersection of matroids corresponds to the stable intersection of their Bergman fans, and that Bergman fans of corank-$1$ matroids are precisely the tropical hyperplanes centered at the origin.
Thus, Theorem~\ref{thm:coradointersections} implies that the stable intersection of a Bergman fan of a matroid with a collection of tropical hyperplanes centered at the origin is the Bergman fan of a certain coRado matroid.
From this, one can recover a special case of a theorem of Fink and Olarte.

\begin{cor}[\cite{fink2022presentations}]
    A matroid is a strict gammoid if and only if its Bergman fan is a stable intersection of tropical hyperplanes centered at the origin. 
\end{cor}

Theorem~7.5 of~\cite{fink2022presentations} states, more generally, that a valuated matroid is a valuated strict gammoid if and only if its associated tropical linear space is a stable intersection of tropical hyperplanes.
The special case, in which the valuations are trivial, is obtained from Theorem~\ref{thm:coradointersections} by letting $M$ be a free matroid.
The matroid $M$ in Example~\ref{ex:R*}, however, is not a strict gammoid, and thus the Bergman fan of $R^*$ is not a stable intersection of tropical hyperplanes.

We now return to the relative nested quotients of a matroid $M$, which are shown in~\cite{backman2023simplicial} to be in correspondence with the monomial bases for the graded pieces of $A^\bullet_\nabla (M)$. 
First, recall the monomial basis for the graded piece of degree $c$ is the set of products $h_{F_1}^{a_1}\cdots h_{F_m}^{a_m}$ of simplicial generators corresponding to nested nonempty flats $F_1, \ldots, F_m \in \mathscr{L}_M$, with each $1 \leq a_i < \text{rk}_M(F_i)-\text{rk}_M(F_{i-1})$ and $\sum a_i = c$.
The coRado matroids in Theorem~\ref{thm:coradointersections} provide a new definition for the {\em relative nested quotients of $M$}: they are matroids of the form $(R_{\graph(\mathcal{A}), N})^*$, where $\mathcal{A}$ is a multiset of flats as described above, with $a_i$ copies of $F_i$ for each $i$.
Furthermore, any principal truncation $T_F(M)$ is given by the dual of the Rado matroid induced by the graph $\graph(\{F\})$ and the matroid $\hat{M}^* \oplus U_{1, \{F\}}$ on $\hat{E} \sqcup \{F\}$.

The graphs $\graph(\mathcal{A})$ in Definition~\ref{defn:Gamma} and the Rado matroids associated to them in Theorem~\ref{thm:coradointersections} can also be used to provide an alternate proof of the Dragon-Hall-Rado theorem of~\cite{backman2023simplicial}.
A collection of (not necessarily distinct) subsets $A_1, \ldots, A_d$ of $E$ is said to satisfy the {\em Dragon-Hall-Rado condition for $M$}, $\dhr(M)$, if $\rk_M(\cup_{j \in J} A_j) \geq |J|+1$ for all nonempty subsets $J$ of $\{1, \ldots, d\}$.

\begin{cor}[Dragon-Hall-Rado theorem~\cite{backman2023simplicial}]
Let $A_1, \ldots, A_d$ be nonempty subsets of a finite set $E$, and let $M$ be a loopless matroid on $E$ of rank $d + 1$.
We have $M \wedge H_{A_1} \wedge \cdots \wedge H_{A_d} = U_{1,E}$ if and only if $\{A_1, \ldots, A_d\}$ satisfies $\dhr(M)$.
\end{cor}

\begin{proof}
Let $\mathcal{A} = \{A_1, \ldots, A_d\}$, let $\graph = \graph(\mathcal{A})$, and let $R = R_{\graph, N}$, where $N$ is the matroid $\hat{M}^* \oplus U_{d, \mathcal{A}}$ on $\hat{E} \cup \mathcal{A}$ (as in Theorem~\ref{thm:coradointersections}).
It follows from Rado's theorem that $\mathcal{A}$ satisfies $\dhr(M)$ if and only if, for any $e \in E$, there is a transversal $I$ of $\mathcal{A}$ such that $e \notin I$ and $I \in \mathcal{I}(M)$~\cite[Proposition~5.2.3]{backman2023simplicial}.
Thus, it suffices to check that $R = U_{|E|-1, E}$ if and only if, for any $e \in E$, $\mathcal{A}$ is matched in $\graph$ to an independent set in $M$ which does not contain $e$.

To prove necessity, suppose that $R = U_{|E|-1, E}$, and let $e \in E$ be arbitrary.
Since $E - e$ is a basis of $R$, it is matched in $\graph$ to a basis $\hat{B}^*$ of $\hat{M}^*$ (of cardinality $|E|-d - 1$) and to each of the vertices $A_1, \ldots, A_d \in V(\graph)$.
Letting $I$ be the set of vertices matched to $A_1, \ldots, A_d$, we see that $I$ is independent in $M$ and does not contain $e$.

To prove sufficiency, suppose that $\mathcal{A}$ satisfies $\dhr(M)$. We prove a stronger claim by induction: for any subset $\mathcal{A'}$ of $\mathcal{A}$, $(R_{\graph(\mathcal{A}'), N'})^*$ is a loopless matroid of rank $d + 1 - |\mathcal{A}'|$, where $N' = \hat{M}^* \oplus U_{|\mathcal{A}'|, \mathcal{A}'}$.

Let $A \in \mathcal{A}$. Since $\mathcal{A}$ satisfies $\dhr(M)$, so does any subset of $\mathcal{A}$; in particular, $\rk_M(A) \geq 2$.
Let $\{a_1, a_2\}$ be a subset of $A$ in $\mathcal{I}(M)$.
We first show that $M \wedge H_A$ is loopless and has rank $d$.
Let $B^* \in \mathcal{B}(M^*)$ be contained in $E - \{a_1, a_2\}$ and let $e \in E$.
If $e \notin B^*$, then $B^* \cup \{a_1\}$ is a basis of $M^* \vee H_A^*$ (of rank $|E| - d$) that does not contain $e$.
Otherwise, since $M^*$ has no coloops, let $B^*{'} \in \mathcal{B}(M^*)$ with $e \notin B^*{'}$.
By the basis exchange axiom, there is an element $f \in B^*{'}$ such that $B^* - e \cup f \in \mathcal{B}(M^*)$.
For at least one $i \in \{1, 2\}$, we have $a_i \neq f$, and thus $B^* - e \cup \{f, a_i\}$ is a basis of $M^* \vee H_A^*$.
Therefore, $M^* \vee H_A^*$ has rank $|E| - d$ and no coloops, implying that $M \wedge H_A$ is loopless of rank $d$.

Now, for a subset $\mathcal{A}' \subseteq \mathcal{A}$ of cardinality at least $2$, let $A \in \mathcal{A}'$, let $e \in E$, and let $I \subseteq E - e$ be a transversal of $\mathcal{A}'$ in $\mathcal{I}(M)$.
Consider the transversal $I - a$ of $\mathcal{A}' - A$.
Since $E-I$ is matched in $\graph(\mathcal{A}')$ to a spanning set of $\hat{M}^*$ and $a$ is matched to $A$, it follows that $E - I \cup a$ spans $M^* \vee H_A^*$, and thus $I - a \subseteq E - e$ is in $\mathcal{I}(M \wedge H_A)$.
Therefore, $\mathcal{A}' - A$ satisfies $\dhr(M \wedge A)$.
The desired result follows by induction and the associativity of matroid intersection.
\end{proof}

Eur and Larson~\cite{eur2023intersection} generalized the simplicial presentation for the Chow ring of a matroid to augmented Chow rings of polymatroids.
In future work, we hope to generalize Theorem~\ref{thm:coradointersections} to the case of polymatroids as well.

\section*{Acknowledgement}

The authors would like to thank Spencer Backman for valuable advice and conversations during the project.
Spencer Backman would in turn like to thank Chris~Eur, Alex~Fink, Jorge~Olarte, Benjamin~Schröter and Connor~Simpson for inspiring discussions.


\begin{thebibliography}{10}

\bibitem{adiprasito2018hodge}
K.~Adiprasito, J.~Huh, and E.~Katz.
\newblock Hodge theory for combinatorial geometries.
\newblock {\em Ann. of Math.}, {\bf 188}(2):381--452, 2018.

\bibitem{backman2023simplicial}
S.~Backman, C.~Eur, and C.~Simpson.
\newblock Simplicial generation of chow rings of matroids.
\newblock {\em J. Eur. Math. Soc. (JEMS)}, 2023.

\bibitem{de1995wonderful}
C.~De~Concini and C.~Procesi.
\newblock Wonderful models of subspace arrangements.
\newblock {\em Selecta Math.}, {\bf 1}:459--494, 1995.

\bibitem{eur2023intersection}
C.~Eur and M.~Larson.
\newblock Intersection theory of polymatroids.
\newblock {\em arXiv preprint arXiv:2301.00831}, 2023.

\bibitem{feichtner2004chow}
E.~M. Feichtner and S.~Yuzvinsky.
\newblock Chow rings of toric varieties defined by atomic lattices.
\newblock {\em Invent. Math.}, {\bf 155}(3):515--536, 2004.

\bibitem{fink2022presentations}
A.~Fink and J.~A. Olarte.
\newblock Presentations of transversal valuated matroids.
\newblock {\em J. Lond. Math. Soc.}, {\bf 105}(1):24--62,
  2022.

\bibitem{hampe2017intersection}
S.~Hampe.
\newblock The intersection ring of matroids.
\newblock {\em J. Combin. Theory Ser. B}, {\bf 122}:578--614, 2017.

\bibitem{larson2022k}
M.~Larson, S.~Li, S.~Payne, and N.~Proudfoot.
\newblock {$ K $}-rings of wonderful varieties and matroids.
\newblock {\em arXiv preprint arXiv:2210.03169}, 2022.

\bibitem{nash1966application}
C.~S. J.~A. Nash-Williams.
\newblock An application of matroids to graph theory.
\newblock In {\em Theory of Graphs, International Symposium}, pages 263--265,
  Rome, 1966.

\bibitem{rado1942theorem}
R.~Rado.
\newblock A theorem on independence relations.
\newblock {\em Q. J. Math.}, {\bf 13}:83--89, 1942.

\bibitem{speyer2008tropical}
D.~E. Speyer.
\newblock Tropical linear spaces.
\newblock {\em SIAM J. Discrete Math.}, {\bf 22}(4):1527--1558, 2008.

\end{thebibliography}
\end{document}